\documentclass[11pt]{amsart}
\usepackage{amsfonts,amssymb,latexsym}
\usepackage{amsmath,amscd}
\usepackage[all]{xy}
\usepackage{pstricks,pst-node}
\usepackage{wrapfig}
\usepackage{float}

\entrymodifiers={+!!<0pt,\fontdimen22\textfont2>}

\makeatletter
\renewcommand{\@seccntformat}[1]{\bf\@nameuse{the#1}.\quad}
\renewcommand\section{\@startsection{section}{1}%
                  \z@{.7\linespacing\@plus\linespacing}{.5\linespacing}%
                  {\normalfont\bfseries \boldmath}}
\renewcommand\subsection{\@startsection{subsection}{2}%
                  \z@{.5\linespacing\@plus.7\linespacing}{-.5em}%
                  {\normalfont\bfseries \boldmath}}
\renewcommand\subsubsection{\@startsection{subsubsection}{3}%
                  \z@{.3\linespacing\@plus.5\linespacing}{-.5em}%
                  {\normalfont\bfseries \boldmath}}
\makeatother

\newtheorem{thm}{Theorem}[subsection] 
 
 \newtheorem{lem}[thm]{Lemma}
 \newtheorem{cor}[thm]{Corollary}
 
 \newtheorem{prop}[thm]{Proposition}

\theoremstyle{definition}
 \newtheorem{example}[thm]{Example}

\numberwithin{equation}{subsection}

\newcounter{txtctr}[section] 

\newcommand{\rk}{\operatorname{rk}}
\newcommand{\Lie}{\operatorname{Lie}}

\newcommand{\ind}{\operatorname{ind}}
\newcommand{\Ext}{\operatorname{Ext}}
\newcommand{\into}{\hookrightarrow}

\renewcommand{\mod}{\operatorname{mod}}

\newcommand{\cal}[1]{\mathcal{#1}}

\newcommand{\Ll}{\cal{L}(\lambda)} 

\newcommand{\hwl}{H^0(w,\lambda)}

\newcommand{\XyP}{X(y)_P}
\newcommand{\V}{\mathcal{V}} 
\newcommand{\NNu}{\mathcal{N}_1(\mathfrak{u})} 

\begin{document}

\title[On the support varieties of Demazure modules]
{\bf On the support varieties of Demazure modules}

\begin{abstract} In \cite{NPV, UGA}, the support varieties for the induced modules/Weyl modules 
for a reductive algebraic group $G$ were computed over the first Frobenius kernel $G_{1}$. A natural 
generalization of this computation is the calculation of the support varieties of Demazure modules 
over the first Frobenius kernel, $B_{1}$,  of the Borel subgroup $B$. In the paper we initiate the study of 
such computations. We complete the entire picture for reductive groups with underlying root 
systems $A_{1}$ and $A_{2}$. Moreover, we give complete answers for Demazure modules 
corresponding to a particular (standard) element in the Weyl group, and provide results relating support varieties 
between different Demazure modules which depends on the Bruhat order.  
\end{abstract}

\author{\sc Benjamin F Jones}
\address
{Department of Mathematics, Statistics, and Computer Science\\ University of Wisconsin-Stout \\
Menomonie\\ WI~54751, USA}
\thanks{Research of the first author was supported in part by NSF
VIGRE grant DMS-0738586.}
\email{jonesbe@uwstout.edu}

\author{\sc Daniel K. Nakano}
\address
{Department of Mathematics\\ University of Georgia \\
Athens\\ GA~30602, USA}
\thanks{Research of the second author was supported in part by NSF
grant DMS-1002135.} \email{nakano@math.uga.edu}


\maketitle

\section{Introduction} 
\subsection{} Let $G$ be a connected, simply connected, simple algebraic group
scheme defined over ${\mathbb F}_{p}$. Moreover, let $W$ be the associated
Weyl group, $B$ a Borel subgroup and $X(T)_{+}$ be the set of dominant
weights. Given $w\in W$ and $\lambda\in X(T)_{+}$, a natural set of
$B$-modules that arise are the Demazure modules labelled by $H^{0}(w,\lambda)$
which can be constructed using iterated inductions involving parabolics
corresponding to simple reflections occuring in a reduced decomposition of
$w$. When $w=w_{0}$ is the long element of $W$ one recovers the induced $G$-modules
$H^{0}(\lambda)=\text{ind}_{B}^{G} \lambda$ which can be realized as global
sections of the line bundle ${\mathcal L}(\lambda)$ over $G/B$.

Demazure modules arise naturally as the global sections on a line bundle
${\mathcal L}(\lambda)$ on the Schubert scheme $X(w)$ \cite[Ch. 14]{Jan}.
The structure of Demazure modules, and $B$-modules with excellent filtration
in general, is closely related to the geometry of the underlying Schubert
varieties (resolution of singularities, sheaf cohomology, normality, and
rational singularities). For example, Mehta and Ramanathan, using the
technique of Frobenius splittings, and later Andersen, using
representation-theoretic techniques, showed that the analog of Kempf's
vanishing theorem holds for sections of a dominant line bundle restricted to
a Schubert variety. This result was applied to complete Demazure's proof of
his character formula. As another example, Polo \cite{Polo} and van der
Kallen \cite{vanderKallen} use the normality of Schubert varieties in a
crucial way in their investigation of the category of $B$-modules with
excellent filtration.

\subsection{} In 2002, at a workshop in Seoul Korea, B. Parshall proposed the
problem of computing the support varieties of the Demazure modules
$H^{0}(w,\lambda)$ over the first Frobenius kernel $B_{1}$. This problem is 
a natural and interesting extension of the ``Jantzen Conjecture" on support varieties 
which predicted the support varieties of $H^{0}(\lambda)$ over $G_{1}$ when the 
characteristic of the field is good. The conjecture was verified by Nakano,
Parshall and Vella \cite{NPV} and the support varieties of $H^{0}(\lambda)$
over $G_{1}$ were shown to be closures of Richardson orbits. This computation 
was later extended to fields of bad characteristic by the University of Georgia VIGRE Algebra Group
\cite{UGA}. In the later case, the support variety of $H^{0}(\lambda)$ is
still irreducible and is the closure of an orbit, but the orbits need not be
Richardson.

Support varieties are natural with respect to the inclusion of $B_{1}$ in
$G_{1}$, so one can deduce from the aforementioned results that the $B_{1}$
support varieties of $H^{0}(\lambda)$ will be unions of closures of \emph{orbital varieties} (see \cite{Mel}). Indeed, orbital varieties should play an important role in the
general theory of support varieties of Demazure modules. This will be more
evident in the results in this paper. 

The main obstacle in computing support varieties for general Demazure modules is that these modules 
are rarely $G$-modules (i.e., their support varieties are not $G$-invariant, and not 
closures of finitely many $G$-orbits). In general there are infinitely many $B$-orbits on 
the nilpotent radical of $\Lie(B)$. At present it is not known how to classify these $B$-orbits. 
The aim of the paper is to study the behavior of support varieties of Demazure modules. 
In many instances we will be able to provide an explicit description of the supports. 

The paper is organized as follows. In Section 2, we present various properties of 
Schubert varieties that will be used throughout the paper. We then discuss properties 
of support varieties over the Frobenius kernels $B_{r}$ and $P_{r}$. Several of the 
main results in \cite{FP} and \cite{NPV} need to be modified and generalized for the purposes of this paper 
(cf. Theorem 3.2.1 and Theorem 3.3.1). In Section 4, we prove a $G$-saturation result for the 
$B_{r}$ support varieties of Demazure modules. In particular, we show that if $w_{1}<w_{2}$ (in 
the Bruhat order) then $G\cdot {\mathcal V}_{B_{r}}(H^{0}(w_{2},\lambda))\subseteq 
G\cdot {\mathcal V}_{B_{r}}(H^{0}(w_{1},\lambda))$. This result is subtle and we indicate by example that this inclusion 
does not hold if one ignores the process of $G$-saturation (cf. Example 4.1.2). With these results, we describe the 
supports of the Demazure modules in the $A_{1}$ case. Calculations of support varieties 
${\mathcal V}_{B_{1}}(H^{0}(w,\lambda))$ are given for specific $w\in W$ in Section 5. Finally, 
in Section 6, we provide a complete description of ${\mathcal V}_{B_{1}}(H^{0}(w,\lambda))$ for algebraic groups of type $A_{2}$. 
An interesting facet of the $A_{2}$-computation is the need to analyze and use information about higher sheaf cohomology groups.

\section{Schubert Schemes}

\subsection{Notation} 

Throughout this paper, let $k$ be an algebraically closed field of characteristic $p > 0$. For an algebraic group $H$, 
the notation $\operatorname{Mod}(H)$ denotes the category of rational $H$-modules and
$\operatorname{mod}(H)$ denotes the category of finite dimensional, rational
$H$-modules.

Let $\Phi$ be a finite irreducible root system for a Euclidean space ${\mathbb E} $.
The inner product on ${\mathbb E} $ will be denoted by $(\ , \ )$.
For $\alpha\in\Phi$,
let
$\alpha^{\vee}=2\alpha/(\alpha,\alpha)$ be the corresponding coroot.
Fix a set $\Delta=\{\alpha_1,\cdots, \alpha_\ell\}$
of simple roots, and let $\Phi^{+}$ be the
corresponding set of positive roots. 
The Weyl group $W\subset O(\mathbb E)$ is the group generated by the reflections
$s_\alpha:{\mathbb E}\to{\mathbb E}$, $\alpha\in\Phi$, given by
$s_\alpha(x)=x-2(x,\alpha^\vee)\alpha$. 

Unless otherwise stated,  $G$ will denote a reductive algebraic group over $k$. 
We will always assume that the derived group $G'$ is simply connected.
Also, assume that $G$ has root system $\Phi$ with respect to a maximal split torus $T$. 
Let $B\supset T$ be the Borel subgroup defined by $-\Phi^{+}$.  The positive
Borel subgroup containing $T$ will be denoted $B^+$. Moreover, 
let $X(T)=X(B)$
be the group of integral characters of $T$ or, equivalently, of $B$.
Given $\lambda\in X(T)$, we will  let $\lambda$
also denote the one-dimensional $B$-module defined by regarding
$\lambda$ as a character on $B$. Then the set of dominant
integral weights is defined by 
$$X_{+}:=X(T)_{+}=\{\lambda\in X(T)\,\,|\,\,\,
0\leq (\lambda,\alpha_i^{\vee}),\quad 1\leq i\leq\ell
\}.$$ 

Let $\rho$ be the half sum of the  positive roots. We partially order
$X(T)$ by setting $\lambda\geq \mu$ if and only if
$\lambda-\mu\in \sum_{\alpha\in \Delta}\mathbb{N}\alpha$. Let $h$ be the
Coxeter number of $G$. Thus, if $G'$ is simple, $h=(\rho,\alpha_0^\vee)+1$
where $\alpha_0$ is the maximal short root in $\Phi$; otherwise,
$h$ is the maximal of the Coxeter numbers for the simple factors
of $G'$.  

Each subset $J \subset \Delta$ gives rise to a standard parabolic
subgroup $P = P_J$ containing $B$ whose Lie algebra is generated by
$\mathfrak{t} = \Lie(T)$, the negative root spaces
$\mathfrak{g}_{-\alpha}$ ($\alpha \in \Phi_+$), and the positive root
spaces in the span of $J$: $\mathfrak{g}_{\alpha}$ for $\alpha \in
\Phi_{J}$. The subgroup $P_J$ has a Levi decomposition $P_J = L_J U_J$
where $\Lie(L_J)$ is generated by $\mathfrak{t}$ and the root spaces
$\mathfrak{g}_{\pm\alpha}$ for $\alpha \in J$ and $\Lie(U_J)$ is
generated by the root spaces $\mathfrak{g}_{-\alpha}$ for $\alpha \in
\Phi_+ \backslash \Phi_{J}$. We denote by $W_J$ the subgroup of $W$
generated by reflections $s_\alpha$ for $\alpha \in J$ and identify it
with the Weyl group of $L_J$. We denote the set of minimal length
right coset representatives for $W/W_J$ by $W^J$. When $P = P_J$ we also use
notations $W_P$ and $W^P$. We denote the opposite parabolic subgroup
that contains $B^+$ by $P_J^+$.

For $G$ as given above, the dominant weights $\lambda\in X(T)_+$ index the simple modules $L(\lambda)$
by their highest weight.  If $\ind_B^G:\mod(B) \to
\mod(G)$ is the induction functor, let $H^0(\lambda)=\ind_B^G \lambda$ for
 $\lambda\in X(T)$.  If $\lambda\notin X(T)_+$, then 
$H^{0}(\lambda)=0$, while if $\lambda\in X(T)_{+}$ then $H^0(\lambda)$
has socle $L(\lambda)$. 

Let $F:G\to G$ be the Frobenius morphism 
on $G$ induced by its ${\mathbb F}_p$-structure.  For
$r\geq 1$, put $G_r={\text{\rm ker}}(F^r)$.  If $H$ is an $F$-stable
subgroup of $G$, write similarly $H_r={\text{\rm
ker}}(F^r|_H)$---e.g., $B_r={\text{\rm ker}}(F^r|_B)$.  
The group scheme $H_r$ is a finite $k$-group, i.e., an affine algebraic group 
scheme over $k$ with finite dimensional coordinate algebra $k[H_r]$. Also, it has
height $\leq r$. In what follows, all affine $k$-groups $A$ will, by definition, be assumed to be algebraic, i.e.,
the coordinate algebra $k[A]$ is assumed to be finitely generated over $k$.
If $M\in \text{Mod}(H)$, let $M^{(r)}$ be the module in $\text{Mod}(H)$ obtained by composing the 
representation corresponding to $M$ with $F^{r}$. 

\subsection{Schubert Schemes}
\label{subsec:schubert}

In this section we follow the notation and conventions of \cite[II. Chapters 13-14]{Jan}.  Fix a parabolic subgroup $P$. The group $G$ has a Bruhat decomposition:
\[ G = \bigcup_{w \in W^P} B \dot{w} P \] where $\dot{w}$ denotes a
chosen representative of $w$ in $N_G(T)$. This induces a decomposition
$G/P = \cup B\dot{w}P/P$ into $B$-stable affine subschemes (cells). We
denote by $X(w)_P$ the closure of the cell $B\dot{w}P/P$ in
$G/P$. These are the Schubert varieties of $G/P$.  When $P=B$ is a
Borel subgroup, we simply use the notation $X(w)=X(w)_B$.

Let $M \in \mod(P)$. The variety $G \times_P M$ is naturally a vector bundle over $G/P$. We denote this vector bundle by $\cal{L}(M)$. The most important case is when $P=B$ and $M = k_\lambda$ for $\lambda \in X(T)$ in which case $\cal{L}(M)$ is a line bundle on $G/B$. If $J \subset  \Delta$ and $\lambda$ satisfies $(\lambda, \alpha^\vee) = 0$ for all $\alpha \in J$, then there is a line bundle 
$\cal{L}(\lambda)_P$  on $G/P$ where $P = P_J$. This bundle pulls back to $\cal{L}(\lambda)$ on $G/B$ under the quotient map $G/B \to G/P$ which is locally trivial. Therefore, by \cite[I 5.17]{Jan}, there is a canonical isomorphism $H^0(G/B, \Ll) \cong H^0(G/P, \Ll_P)$.

The cohomology groups $H^i(G/B, \cal{L}(M))$ are naturally $G$-modules. 
For each $y \in W^P$ the inclusion $X(y)_P \into G/P$ induces the restriction map $H^i(G/P, \cal{L}(M)) \to 
H^i(X(y)_P, \cal{L}(M))$.   

The schemes $X(y)_P$ admit resolutions of singularities $\phi: X \to X(y)_P$ which are equivariant with respect to $B$ and depend on a reduced decomposition of $\dot{y}$, a minimal length coset representative of $y$ in $W$ (cf. \cite[13.6]{Jan}). The resolution $X$ is defined as a subset of a variety $Z$ which is a fiber bundle over $G/B$:

\begin{equation} \label{fig:resolution-diagram}
\xymatrix{
X \ar[r] \ar[d]_{\dot{\phi}} & Z \ar[d] \\
X(\dot{y}) \ar@{^{(}->}[r] \ar[d]_{\pi_P} & G/B \ar[d]^{\pi_P} \\
X(y)_P \ar@{^{(}->}[r]_i & G/P \quad .}
\end{equation}
In the diagram, $\pi_P$ is the natural projection $G/B \to G/P$ which is birational when restricted to $X(\dot{y})$ and the resolution $\phi$ is $\phi = \pi_P \circ \dot{\phi}$.
We need the following well known geometric results on Schubert varieties and sheaf cohomology.

\begin{prop} \label{prop:jan-cohomology} \cite[II. 14.15]{Jan} Let
  $y \in W^P$, let $\dot{y}$ be a minimal length right coset representative
  of $y$ in $W$, and let $w \in W$. Then the following hold:
\begin{enumerate}
\item $X(y)_P$ is normal, closed subscheme of $G/P$.
\item For every vector bundle $V$ on $G/P$ and $i\geq 0$, $H^i( X(y)_P, V ) \cong 
H^i( X(\dot{y}), \pi_P^*V ) \cong H^i( X, \phi^*V ).$
\item For all $\lambda \in X(T)_+$, $H^i( X(w), \Ll ) = 0$ for $i>0$.
\item Given $\lambda \in X(T)_+$ such that $( \lambda, \alpha^\vee ) = 0$ for all $\alpha \in J$ where $J\subseteq \Delta$, 
the restriction map $H^i( G/P, \Ll_P ) \to H^i( \XyP, \Ll_P )$ is surjective and moreover 
$$H^i( \XyP, \Ll_P ) = 0$$ for all $i > 0$ where $P = P_J$ is the standard parabolic subgroup associated to $J$.
\end{enumerate}
\end{prop}

We also need the identification of the $G$-module $H^i(G/P,\cal{L}(M)_P)$ with induction from $P$ to $G$.

\begin{prop} \cite[I.5.12]{Jan}
\label{prop:ind-vs-cohomology}
\begin{enumerate}
\item For any $P$-module $M$ and $i\geq 0$ there is a canonical isomorphism
\[ R^i \ind^G_P M \cong H^i( G/P, \cal{L}(M) ) .\]
\item Let $H \subset K$ be $k$-group schemes such that $K/H$ is Noetherian (e.g.,  $K$ is reductive and $H$ is a parabolic, or $H \subset K \subset G$ are both parabolic in a reductive group) and let $M$ be a rational $H$-module. Then,
\[ R^i \ind^K_H M = 0 \]
for $i > \dim K/H$.
\end{enumerate}
\end{prop}

\section{Support Varieties over $P_{r}$} 

\subsection{} In this section let $A$ be an arbitrary 
finite $k$-group scheme and $\text{mod}(A)$ be the category of finite-dimensional 
$A$-modules. We will consider maximal ideals in the commutative 
part of the cohomology ring so set 
$$R:=\operatorname{H}(A,k)=\begin{cases} \operatorname{H}^{2\bullet}(A,k) & \text{if $\text{char }k\neq 2$} \\ 
                       \operatorname{H}^{\bullet}(A,k)  & \text{if $\text{char }k=2$}.  \\ 
\end{cases} 
$$ 
Friedlander and Suslin \cite{FS} proved that $R$ is a finitely generated $k$-algebra \cite{FS}.
Let $\V_{A}$ denote the variety associated to the maximum ideal spectrum of $R$. 
Given $M,M'\in \text{mod}(A)$ we define the {\em relative support variety} $\V_A(M,M')
=\text{Maxspec}(R/J_{M,M^{\prime}})$ where $J_{M,M^{\prime}}$ is the annhilator 
of the action of $R$ on $\text{Ext}^{\bullet}_{A}(M,M^{\prime})$. The action (Yoneda product) of 
$R=\text{Ext}_{A}^{\bullet}(k,k)$ on $\text{Ext}^{\bullet}_{A}(M,M^{\prime})$ is given by taking an 
extension in $R$ applying $- \otimes_{k} M^{\prime}$ then concatenating the new class with 
an extension class in $\text{Ext}^{\bullet}_{A}(M,M^{\prime})$ (cf. \cite[Section 2.6]{Ben}).

The ordinary {\em support variety} of $M\in \text{mod}(A)$ is $\V_{A}(M):=\V_{A}(M,M)$. 
In general for any $M,M'\in \text{mod}(A)$, $\V_A(M,M')$ is a homogeneous closed subvariety 
contained in $\V_{A}=\V_{A}(k)$. For the basic properties of support varieties 
for finite $k$-group schemes we refer the reader to \cite[Section 5]{FPe} and \cite[\S2.2]{NPV}. 

Let $H$ be a closed subgroup of a finite $k$-group $A$ of height $\leq r$.  
Suslin, Friedlander and Bendel \cite[(5.4)]{SFB2} proved that  
the image of the restriction map $\text{res}:\text{H}(A,k)_{\text{red}}\to \text{H}(H,k)_{\text{red}}$
contains all $p^r$th powers $x^{p^r}$ of elements $x\in H(H,k)_{\text{red}}$, and the
induced morphism $\text{res}^{*}:\V_H\to \V_A$ maps $\V_H$ homeomorphically onto its image as 
a closed subvariety of $\V_A$. In this paper we will identify the image of 
$\V_{H}$ with $\text{res}^{*}(\V_{H})$ in $\V_{A}$. Under this map we have the 
following naturality property. 

\begin{prop}
\label{prop:naturality} 
Let $H$ be a closed subgroup of $A$. Then $\V_{H}(M) =\V_{H}\cap \V_{A}(M)$.
\end{prop}

For infinitesimal group schemes of height 1, one can make the descriptions 
of support varieties quite explicit. Let $H$ be an affine algebraic group scheme 
defined over ${\mathbb F}_p$, $H_{1}=\text{ker }H_{1}$, and $\mathfrak{h}=\text{Lie }H$ 
(which is a restricted Lie algebra with $[p]$ operator). Let ${\cal N}_{1}({\mathfrak{h}})$  
be the closed subvariety of nilpotent elements in $\mathfrak h$ of $H$ defined by 
$${\cal N}_1(\mathfrak{h}):=\{x\in{\mathfrak h}\,|\,x^{[p]}=0\}.$$ 
We have following identification of varieties: 

\begin{prop}
\label{prop:VH1} \cite[(1.6), (5.11)]{SFB1}
$\V_{H_1}$ is homeomorphic to ${\cal N}_1(\mathfrak{h})$.
\end{prop}

Finally, we can use the identification in (3.1.2) to identify $\V_{H_{1}}(M)$ as a closed subvariety of ${\cal N}_1(\mathfrak{h})$. 

\begin{prop}
\label{prop:VH1M} \cite[(1.3) Theorem]{FP}
$\V_{H_1}(M)$ is homeomorphic to $$\{x\in {\mathcal N}_1({\mathfrak h}):\ M\ \text{is not } 
x\text{-projective} \} \cup\{0\}.$$
\end{prop}

\subsection{} For the purposes of this paper we need to analyze the relationship of 
support varieties over $B_{r}$ versus $P_{r}$ where $P$ is a parabolic subgroup of $G$. 
The following result is a generalization of \cite[(1.2) Theorem]{FP} and \cite[Proposition 4.5.2]{Be}. 

\begin{thm} 
\label{thm:fp-bend-generalization}
Let $J\subseteq \Delta$, $P=P_{J}$ be the associated parabolic subgroup, 
and $M\in \operatorname{mod}(P)$. Then 
$$\V_{P_{r}}(M)=P\cdot \V_{B_{r}}(M).$$ 
\end{thm} 

\begin{proof} The proof follows along the same line of reasoning as in \cite[(1.2) Theorem]{FP}. 
We will indicate what modifications are necessary. 
Let $\Psi=\text{res}^{*}:\V_{B_{r}}(M)\rightarrow 
{\mathcal V}_{P_{r}}(M)$ be the map on varieties induced from the 
restriction map $\text{res}:\text{H}^{\bullet}(P_{r},k)\rightarrow 
\text{H}^{\bullet}(B_{r},k)$. According to 
\cite[(1.6), (5.11)]{SFB1}, we can identify ${\mathcal V}_{B_{r}}(M)$ with 
$\Psi({\mathcal V}_{B_{r}}(M))$ in ${\mathcal V}_{P_{r}}(M)$. Since ${\mathcal V}_{P_{r}}(M)$ is invariant 
under $P$ we have 
$$P\cdot {\mathcal V}_{B_{r}}(M)\subseteq {\mathcal V}_{P_{r}}(M).$$ 
We need to show that the reverse inclusion holds. 

Following the proof of \cite[(1.2) Theorem]{FP}, set 
\begin{eqnarray*} 
I_{M}&=&\text{ker}\{\text{H}^{\bullet}(B_{r},k)\rightarrow \text{Ext}^{\bullet}_{B_{r}}(M,M)\}\\ 
J_{M}&=&\text{ker}\{\text{H}^{\bullet}(P_{r},k)\rightarrow \text{Ext}^{\bullet}_{P_{r}}(M,M)\}\\ 
K_{M}&=&\{x\in \text{H}^{\bullet}(P_{r},k):\ p\cdot \text{res}(x)\in I_{M} \ \forall p\in P\}\\ 
L_{M}&=&\{x\in \text{H}^{\bullet}(P_{r},k):\ p\cdot \text{res}(x)\in \sqrt{I_{M}} \ \forall p\in P\}
\end{eqnarray*} 

Now replace ``$G$'' by ``$P$'', remove the ``symmetric algebras'', and use the fact 
that $H^{m}(P/B,-)=0$ for $m> \dim P/B$. Then we can conclude that 
$K_{M}\subseteq \sqrt{J_{M}}$, thus ${\mathcal V}_{P_{r}}(M)\subseteq  
P\cdot {\mathcal V}_{B_{r}}(M)$.    
\end{proof}

\subsection{}
\label{subsec:npv-541}
For $M$ a rational $B$-module, the relationship between the (relative) $B_r$ support variety of a module induced from $M$ and the $G_r$ support variety is described in \cite[Theorem 5.4.1]{NPV}. We generalize this result to the parabolic case as follows.

\begin{thm} \label{thm:npv-541}
Let $M$ be a rational $B$-module and $P$ be a parabolic subgroup of $G$ which contains $B$. Suppose that $R^m \ind^P_B M = 0$ for $m \ne t$, where $t$ is a fixed integer. Then,
\[ \V_{P_r}( R^t \ind^P_B M ) = P \cdot \V_{B_r}( R^t \ind^P_B M, M) .\]
\end{thm}

\begin{proof}
The proof of \cite[Theorem 5.4.1]{NPV} is formal and carries over after replacing $G$ by $P$. 
The main issue involves the use of a spectral sequence which in our case is:
\[ 
E_2^{m,n}=R^m\ind^{P/P_r}_{B/B_r}\Ext^n_{B_r}(R^t 
\ind_B^P M,M)\Rightarrow
\Ext^{m+n-t}_{P_r}(R^t\ind_B^P M,R^t\ind_B^P M), 
\]
and an increasing filtration whose finiteness depends on a vanishing result,
\[ R^m \ind^{P/P_r}_{B/B_r} = 0 \quad \text{for} \quad m > \dim P/B. \]
This vanishing result holds by Proposition \ref{prop:ind-vs-cohomology}(ii).
\end{proof}

\section{$G$-Saturation}

\subsection{} We are interested in determining the support varieties $\V_{B_1}(H^0(X(w), \Ll))$ for all $w \in W$ 
and $\lambda \in X_+$. In particular, we want to understand the inclusion relations among support varieties for different $w$ and $\lambda$ 
of particular interest. In some instances we will use $H^{0}(w,\lambda):=H^{0}(X(w),{\mathcal L}(\lambda))$ as 
a short hand notation. In the following theorem, we prove that for a fixed weight $\lambda$, the inclusion relation on 
the $G$-saturation of support varieties for Demazure modules respects the Bruhat order on $W$. 

\begin{thm}
\label{thm:saturation}
Let $\lambda \in X_+$ and $w_1 < w_2 $ in the Bruhat order on $W$. Then,
\[
G \cdot \V_{B_r}(H^0(w_2,\lambda)) \subseteq G \cdot \V_{B_r}(H^0(w_1,\lambda)) .
\]
\end{thm}

\begin{proof}
By induction on $\ell(w_2) - \ell(w_1)$, it suffices to prove the result when 
$w_2 = s_\alpha w_1$ and $\ell(w_2) = \ell(w_1) + 1$. Let $P_\alpha$ be the minimal parabolic corresponding to 
$\alpha$. By Theorem~\ref{thm:fp-bend-generalization},
\begin{equation}
\label{eq:sat1}
\V_{(P_\alpha)_r}( H^0(w_2,\lambda) ) = P_\alpha \cdot \V_{B_r}( H^0(w_2,\lambda) ) .
\end{equation}
Since $H^0(w_2,\lambda) \cong \ind^{P_\alpha}_B H^0(w_1,\lambda)$, Theorem \ref{thm:npv-541} implies:
\begin{equation}
\label{eq:sat2}
\begin{array}{ll}
P_\alpha \cdot \V_{B_r}( H^0(w_2, \lambda) ) & =  P_\alpha \cdot \V_{B_r}( H^0(w_2, \lambda) , H^0(w_1, \lambda) )  \\
& \subseteq P_\alpha \cdot \V_{B_r}( H^0(w_1, \lambda) ).
\end{array}
\end{equation}
Combining \eqref{eq:sat1} and \eqref{eq:sat2} we have
\[ \V_{(P_\alpha)_r}( H^0(w_2, \lambda) ) \subseteq P_\alpha \cdot \V_{B_r}( H^0(w_1, \lambda) ) ,\]
so acting by $G$ on both sides we certainly have:
\[ G \cdot \V_{(P_\alpha)_r}( H^0(w_2, \lambda) ) \subseteq G \cdot \V_{B_r}( H^0(w_1, \lambda) ) .\]
Finally, by (3.1.1) $\V_{B_r}(M) \subseteq \V_{(P_\alpha)_r}(M)$ for all $M \in \mod(P_\alpha)$. Thus,
\[ G \cdot \V_{B_r}( H^0(w_2, \lambda) ) \subseteq G \cdot \V_{B_r}( H^0(w_1, \lambda) ) .\]
\end{proof}

\subsection{} We should remark that the result above is rather subtle in the sense that inclusion of the 
$B_{1}$-support varieties of Demazure modules need not be preserved under the Bruhat order. 
This can be seen in the following example. 

\begin{example}
\label{ex:A2-steinberg}
Let $p\geq 3$, $\lambda = (p-1)\rho$ (the Steinberg weight), and $G = SL(3)$.
Let ${\mathfrak u}_{\alpha}$ (resp. ${\mathfrak u}_{\beta}$) be the 
unipotent radical of the Lie algebra of $P_{\alpha}$ (resp. $P_{\beta}$). 
The computation in Section \ref{subsec:A2} gives the support varieties 
$\V_{B_1}(H^0(w, (p-1)\rho))$ for all $w \in W$, see Table 1 (below).

\begin{table}[H]
\label{tab:A2-steinberg}
\begin{tabular}{c|c}
$w$ & $\V_{B_1}(H^0(w, (p-1)\rho))$ \\
\hline
$e$        & $\mathfrak{u}$ \\
$s_\alpha$ & $\mathfrak{u}_\alpha$ \\
$s_\beta$  & $\mathfrak{u}_\beta$ \\
$s_\alpha s_\beta$ & $\mathfrak{u}_\alpha \cup \mathfrak{u}_\beta$ \\
$s_\beta s_\alpha$ & $\mathfrak{u}_\alpha \cup \mathfrak{u}_\beta$ \\
$w_0$ & $\{0\}$
\end{tabular}
\vspace{10pt}
\caption{Support varieties for Demazure modules in type $A_2$ with highest
weight $(p-1)\rho$.}
\end{table}

In particular, the pair $s_\beta$ and $s_\alpha s_\beta$ illustrate that $w_1 <
w_2$ does not neccesarily imply that $\V_{B_1}(H^0(w_2,\lambda)) \subseteq
\V_{B_1}(H^0(w_1,\lambda))$. Note, however, that the saturations in these two
cases agree:
\[ G \cdot \mathfrak{u}_\beta = G \cdot \mathfrak{u}_\alpha = G \cdot
(\mathfrak{u}_\alpha \cup \mathfrak{u}_\beta) .\]
\end{example}

\subsection{The Regular Case}

Fix a dominant weight $\lambda$. The subset 
\[ \Phi_{\lambda, p} = \{ \alpha \in \Phi^+ \mid (\lambda + \rho, \alpha^\vee) \in p\mathbb{Z} \} \]
is a subroot system of $\Phi$ which, when the prime $p$ is good relative to $\Phi$, and conjugate under the Weyl group to a
root system $\Phi_I$ spanned by a subset $I \subseteq \Delta$ of simple roots, see \cite[Prop. 24, pg. 165]{Bo}. The weight $\lambda$ is called \emph{$p$-regular} if $\Phi_{\lambda,p} = \emptyset$. 

\begin{prop} \label{prop:regular-case} Let $\lambda$ be a $p$-regular weight in $X_{+}$, 
then $\V_{B_1}(H^0(w,\lambda)) =\V_{B_{1}}$. 
\end{prop}

\begin{proof} If $w_0$ denotes the
  longest element of the Weyl group, then $w \le w_0$ so Theorem
  \ref{thm:saturation} gives us an inclusion of the saturated
  supports:
\[ G \cdot \V_{B_1}(H^0(w_0, \lambda)) \subseteq G \cdot
\V_{B_1}(H^0(w, \lambda)).\] Since $X(w_0) = G/B$, $H^0(w_0, \lambda)$
is a $G$-module and we have 
$$\V_{G_1}(H^0(w_0, \lambda)) =\V_{G_1}(H^0(G/B, \Ll)).$$  
Moreover, by \cite[(1.2) Theorem]{FP},
$\V_{G_1}(H^0(G/B, \Ll)) = G \cdot \V_{B_1}(H^0(G/B, \Ll))$.  Putting
these results together, we have
\begin{equation} 
\label{eq:regular-case-eq-1}
\V_{G_1}(H^0(w_0, \lambda)) \subseteq  G \cdot \V_{B_1}(H^0(w,
\lambda)) .
\end{equation}
Since $\lambda$ is $p$-regular, $\V_{G_1}(H^0(w_0, \lambda)) = \V_{G_{1}}$ by 
\cite[Proposition (4.1.2)]{NPV} and thus 
\[  \V_{G_{1}} \subseteq G \cdot \V_{B_1}(H^0(w, \lambda)) \subseteq \V_{G_{1}}.\]
Therefore, we must have 
\[ G \cdot \V_{B_1}(H^0(w, \lambda))=\V_{G_{1}}. \]
Since $\lambda$ is $p$-regular we have $p\geq h$ (cf. \cite[6.2 (9)]{Jan}), and 
$\V_{G_{1}}$ identifies with the nilpotent cone in ${\mathfrak g}$. Therefore, 
the closed conical $B$-stable variety $\V_{B_1}(H^0(w, \lambda))$ must contain a regular nilpotent element.  
It follows that 
$$\V_{B_1}(H^0(w, \lambda))={\mathfrak u}={\mathcal N}_{1}({\mathfrak u})=\V_{B_{1}}.$$ 
\end{proof}

\subsection{The Root System $A_{1}$}
\label{subsec:A1}

We conclude this section by illustrating Proposition
\ref{prop:regular-case} in the situation when the root
system $\Phi$ is $A_1$ (i.e., for the group $G = \operatorname{SL}(2)$).

Let $G = \operatorname{SL}(2)$ and $\lambda$ be a dominant integral
weight (represented by a non-negative integer). In
this case $G/B \cong \mathbb{P}^1$ and $W = \{ e, s_\alpha \}$. 

Let $w=e$, we have $X(e) = eB/B \cong \{ \text{pt}. \}$. It follows
that $\dim H^0(w, \lambda) = 1$ and so by the rank variety 
description, $\V_{B_1}(H^0(w,\lambda)) = {\mathfrak u}$ which is  
independent of $\lambda$.

The case $w = s_\alpha$ is the long element of $W$ so we have
$X(s_\alpha) = G/B$. Now the weight $\lambda$ is $p$-regular if
and only if $p \nmid \lambda + 1$. So by Proposition 
\ref{prop:regular-case}, $\V_{B_1}(H^0(s_\alpha, \lambda)) = \V_{B_1}
= \mathfrak{u}$ unless $p \mid \lambda + 1$. When $p \mid \lambda +
1$, a simple application of \cite[Theorem 6.2.1]{NPV} gives
\[ \V_{B_1}(H^0(w, \lambda ) ) =  \{ 0 \} .\]

We summarize the situation for type $A_1$ in Table \ref{tab:A1}.

\begin{table}[!ht]
\begin{tabular}{|c||c|c|c|}
  \hline
  & $w$ & & \\
\hline
\hline
$\mathbf{\lambda}$ &    & $p \nmid \lambda+1$ & $p \mid \lambda+1$ \\
\hline
& $e$ & $\mathfrak{u}$     & $\mathfrak{u}$ \\
\hline
& $s_\alpha$ & $\mathfrak{u}$ & $\{0\}$ \\
\hline
\end{tabular}

\bigskip
\caption{Calculation of support varieties for all Demazure modules in
  type $A_1$.}
\label{tab:A1}
\end{table}

\section{Calculation of Support Varieties}
\label{sec:calc}

In this section we determine the support varieties of Demazure modules for arbitrary 
reductive groups $G$ when the underlying Schubert scheme corresponds either to the longest element in $W_I$ (for any $I\subseteq\Delta$) or to the longest element in $W^J$ (for certain subsets $J \subseteq \Delta$).

\subsection{Long elements in $W^J$} Let $\lambda \in X_+$ and define 
\[ J_\lambda = \left\{ \alpha \in \Delta \mid  \langle \lambda, \alpha^\vee \rangle = 0 \right\}. \]
For any subset $J \subseteq \Delta$, let $w_{0,J}$ denote
the Weyl group element of maximal length in $W^J$. 

\begin{prop} \label{prop:w-is-w0J}  There is an isomorphism of $B$-modules
  \[ H^0(X(w_{0,J_\lambda})_{P_{J_\lambda}}, \Ll) \cong H^0(G/B, \Ll) .\]
\end{prop}

\begin{proof}
  For simplicity, let $w = w_{0,J_\lambda}$ and $P = P_{J_\lambda}$ for the remainder of this section.
The resolution diagram \eqref{fig:resolution-diagram} induces a diagram involving cohomology groups:
\begin{equation} \label{fig:cohomology-diagram}
\xymatrix{
H^0(X(\dot{w}),\Ll) & H^0(G/B,\Ll)  \ar[l]_{j^*} \\
H^0(X(w)_P,\Ll_P) \ar[u]_{(\pi_P\lvert_{X(\dot{w})})^*} & H^0(G/P,\Ll_P) \ar[l]_{i^*}  \ar[u]_{\pi_P^*}} .
\end{equation}
By Proposition \ref{prop:jan-cohomology}, $(\pi_P\lvert_{X(\dot{w})})^*$ is an
isomorphism. Also, the choice of $w$ implies that $X(w)_P = G/P$, hence $i$
is the identity and $i^*$ is an isomorphism. By local triviality (cf. \cite[I
5.17]{Jan}), the map $\pi_P^*$ is an isomorphism. The diagram
commutes, thus $j^*$ is an isomorphism.
\end{proof}

The proposition and \cite[Theorem 6.2.1]{NPV} allows us to identify the support
variety of $H^0(X(w)_P, \Ll_P) $ in this special case. 
Choose $x \in W$ such that $x(\Phi_{\lambda,p}) = \Phi_I$ 
for some subset $I \subseteq \Delta$.

\begin{thm} \label{thm:w-is-w0J} With $J=J_\lambda$, $w=w_{0,J}$, and $P = P_{J}$ as above, 
\[ {\mathcal V}_{B_1}( H^0(X(w)_P, \Ll_P) ) = ( G\cdot \mathfrak{u}_I )
\cap \NNu .\] 
\end{thm}

\begin{proof} By \cite[(6.2.1) Theorem]{NPV}, 
${\mathcal V}_{G_1}( H^0(G/B, \Ll) ) =G\cdot\mathfrak{u}_I$. The isomorphism of 
Proposition \ref{prop:w-is-w0J} along with naturality of supports, see (3.1.1), implies that 
 \begin{align*} {\mathcal
V}_{B_1}( H^0(X(w)_P, \Ll_P) ) & = {\mathcal V}_{B_1}( H^0(G/B, \Ll) ) \\ & =
{\mathcal V}_{G_1}( H^0(G/B, \Ll) ) \cap \NNu \\ & = ( G\cdot \mathfrak{u}_I ) \cap
\NNu . \end{align*} 
\end{proof}

Theorem~\ref{thm:w-is-w0J} implies that the $B_1$ support varieties of
certain Demazure modules are unions of the closures of orbital varieties. Recall from the introduction 
that the $B_1$ support varieties of induced modules $H^0(G/B, \Ll)$ 
are also unions of orbital variety closures. It remains an interesting open problem 
whether or not the support varieties of all Demazure modules are unions of orbital variety 
closures and whether one can realize all such closures as support varieties of certain
modules.

\subsection{Longest element in $w_{I}$} In this section, let $I\subseteq \Delta$ be an arbitrary subset and let 
$w=w_{I}\in W_{I}$ such that $w_{I}(\alpha)<0$ for all $\alpha\in I$. The
element $w_{I}$ is the long element in the group $W_{I}$. First, note that in
this case by \cite[II 13.3 (4)]{Jan} 
\begin{equation*}
H^{0}(X(w),{\mathcal L}(\lambda))\cong \text{ind}_{B}^{P_{I}}\lambda .
\end{equation*} 
Consequently, $H^{0}(X(w),{\mathcal L}(\lambda))$ is a
$P_{I}$-module with $U_{I}$ acting trivially. The following theorem describes
the support variety of $H^{0}(X(w),{\mathcal L}(\lambda))$ as a
$(P_{I})_{1}$-module by reducing down to case of \cite[Theorem 6.2.1]{NPV} for
the Levi subgroup $L_{I}$.

\begin{thm} Let $I\subseteq \Delta$ with ${\mathfrak u}_{I}=\operatorname{Lie
}U_{I}$, and $w=w_{I}$. Then 
$$\V_{(P_{I})_{1}}(H^{0}(X(w),\Ll))=
[\V_{(L_{I})_{1}}(H_{I}^{0}(\lambda))+{\mathfrak u}_{I}] \cap {\mathcal N}_{1}({\mathfrak p}_{I}).$$ 
\end{thm}

\begin{proof} 
Set ${\mathfrak l}_{I}=\text{Lie }L_{I}$ and ${\mathfrak
u}_{I}=\operatorname{Lie }U_{I}$. First observe by \cite[(4.2) Examples]{CPS}
that 
\begin{equation*} \text{ind}_{B}^{P_{I}}\lambda|_{L_{I}}\cong
\text{ind}_{L_{I}\cap B}^{L_{I}}\lambda:=H^{0}_{I}(\lambda). 
\end{equation*}
Let $z=x+y$ where $x\in {\mathfrak l}_{I}$, $y\in {\mathfrak u}_{I}$ and $z\in
{\mathcal N}_{1}({\mathfrak p}_{I})$. Then by \cite[Proposition 5.2(a)]{CLNP}, 
$x\in {\mathcal N}_{1}({\mathfrak l}_{I})$. Since ${\mathfrak u}_{I}$ acts trivially 
on $H^{0}_{I}(\lambda)$ we have 
\begin{equation*}\label{E:compareim}
z.H^{0}_{I}(\lambda)=x.H^{0}_{I}(\lambda). 
\end{equation*} 
In particular, $H^0_I(\lambda)$ is $z$-projective  if and only if
it is $x$-projective. By
the realization of the support varieties in terms of rank varieties, we can
conclude that $z\in {\mathcal V}_{(P_{I})_{1}}(H^{0}(X(w),{\mathcal
L}(\lambda)))$ if and only if $x\in {\mathcal
V}_{(L_{I})_{1}}(H^{0}(X(w),{\mathcal L}(\lambda)))$.

Therefore, 
\begin{equation*} {\mathcal V}_{(P_{I})_{1}}(H^{0}(X(w),{\mathcal
L}(\lambda))) = [{\mathcal
V}_{(L_{I})_{1}}(H_{I}^{0}(\lambda))+{\mathfrak u}_{I}] \cap
{\mathcal N}_{1}({\mathfrak p}_{I}). 
\end{equation*} 
\end{proof}

Using \cite[Theorem 6.2.1]{NPV} we obtain the following description of the
support variety. Recall that when the prime $p$ is good there exists $x \in W$ such that $x(\Phi_{\lambda,p}) =
\Phi_J$ for some subset $J \subseteq \Delta$. 

\begin{cor} \label{cor:w-is-wI}
Let $w = w_I$ as above, let $\lambda \in X(T)_+$, and suppose $p$ is a good prime for $\Phi$. Let $x \in W_I$ be such that $x( (\Phi_I)_{\lambda,p} ) = (\Phi_I)_J$ for some
subset $J \subset I$. Let $\mathfrak{u}_{I,J}$ be the nilradical of the
parabolic in $\mathfrak{l}_I$ corresponding to $J$. Then, 
\begin{eqnarray*}
\V_{(P_{I})_{1}}(H^{0}(X(w),\Ll)) & = & (L_I \cdot \mathfrak{u}_{I,J} + {\mathfrak u}_{I})\cap 
{\mathcal N}_{1}({\mathfrak p}_{I}) \\
 & = & \left( L_I \cdot (\mathfrak{u}_{I,J} + {\mathfrak u}_{I}) \right) \cap {\mathcal N}_{1}({\mathfrak p}_{I}) \\
 & = & \left( L_I \cdot \mathfrak{u}_J \right) \cap {\mathcal N}_{1}({\mathfrak p}_{I}). 
\end{eqnarray*}
\end{cor} 

\subsection{Parabolic Upper and Lower Bounds} 

The explicit calculation of Corollary \ref{cor:w-is-wI} and the inclusions among saturated support varieties in Theorem \ref{thm:saturation} give upper and lower bounds on the saturation $G \cdot \V_{B_1}(\hwl)$ for arbitrary $w \in W$ and $\lambda \in X(T)_+$. To state the bounds obtained we introduce some notation. For $v \in W$, let 
$v = s_{\gamma_1} \cdots s_{\gamma_n}$ be a reduced expression. Define the support of $v$ by $S(v) = \{ \gamma_1, \ldots, 
\gamma_n \}$. This definition is independent of the reduced expression chosen (cf. \cite[Theorem 3.3.1]{Bj}). As in the previous section, $w_I$ denotes the long element of $W_I$ for a subset $I \subseteq \Delta$.

\begin{lem}
\label{lem:w-bounds}
If $v \in W$ then $v \le w_{S(v)}$. Moreover, $v \le w_I$ implies that $S(v) \subseteq I$ and $w_{S(v)} \le w_I$.
\end{lem}
\begin{proof}
This is a consequence of \cite[Theorem 5.10]{Hum}.
The expression $v = s_{\gamma_1} \cdots s_{\gamma_n}$ implies that $v \in W_{S(v)}$ and hence $v \le w_{S(v)}$ since the latter is the unique longest element of 
$W_{S(v)}$. Similarly, $v \le w_I$ implies that the generators of $W_{S(v)}$ are contained in $W_I$ hence $w_{S(v)} \le w_I$.
\end{proof}

The lemma gives us a precise characterization of the least upper bound by 
elements of the form $w_I$ where $I \subseteq \Delta$. 
In general there is no unique greatest lower bound as Example \ref{ex:no-unique-lower-bound} shows.

\begin{example}
\label{ex:no-unique-lower-bound}
Let $W$ be the Weyl group of type $A_3$ generated by simple reflections $s_{\alpha_1}, s_{\alpha_2}, s_{\alpha_3}$ 
such that $s_{\alpha_1}$ and $s_{\alpha_{3}}$ commute. 
\begin{itemize}
\item The element $w = s_{\alpha_{1}}s_{\alpha_{2}}$ has support $S(s_{\alpha_{1}}s_{\alpha_{2}}) = 
\{s_{\alpha_{1}},s_{\alpha_{2}}\}$ and its unique parabolic upper bound in the Bruhat order is 
$w_{\{\alpha_1,\alpha_2\}} = s_{\alpha_{1}}s_{\alpha_{2}}s_{\alpha_{1}}$. On the other hand, $w$ has 
maximal lower bounds 
$w_{\{\alpha_1\}} = s_{\alpha_{1}}$ and $w_{\{\alpha_2\}} = s_{\alpha_{2}}$ which are incomparable.   
\item The element $w = s_{\alpha_{1}}s_{\alpha_{2}}s_{\alpha_{3}}$ has support 
$S(s_{\alpha_{1}}s_{\alpha_{2}}s_{\alpha_{3}}) = \Delta$ so its unique upper bound is 
$w_0 = s_{\alpha_1}s_{\alpha_2}s_{\alpha_3}s_{\alpha_1}s_{\alpha_2}s_{\alpha_3}$. 
Moreover, $w$ has a unique maximal lower bound given by $w_{\{\alpha_1,\alpha_3\}} = 
s_{\alpha_1}s_{\alpha_3}$. The set of all 
parabolic elements bounded above by $w$ is
$\{ e, s_{\alpha_{1}}, s_{\alpha_{2}}, s_{\alpha_{3}}, s_{\alpha_{1}}s_{\alpha_{3}} \}$.
\end{itemize}
\end{example}

As an application, the explicit description of supports given by Corollary \ref{cor:w-is-wI} implies the 
following explicit upper and lower bounds on the $G$-saturated support variety of a Demazure module.

\begin{prop}
\label{prop:support-lower-bound}
Let $v \in W$ and $\lambda \in X(T)_+$ then,

\begin{equation*}
  G \cdot \V_{B_1}(H^0(w_{S(v)}, \lambda ) )  \subseteq  G \cdot \V_{B_1}(H^0(v, \lambda ) ) \subseteq \bigcap\limits_{w_I \le v}  G \cdot \V_{B_1}(H^0(w_I,\lambda))
\end{equation*}
where the intersection may be taken over the set of $w_I \le v$ which are 
maximal with respect to that property.
\end{prop}
Recall that the varieties of the form $\V_{B_1}(H^0(w_I,\lambda))$ are explicitly determined in Corollary \ref{cor:w-is-wI}.

\section{Support varieties of Demazure modules for the root system $A_2$}
\label{sec:A2}

In this section we present explicit calculations of the support
varieties for Demazure modules when the group $G$ has a root system 
of type $A_2$. We proceed by applying our results from Section
\ref{sec:calc} in the case when the prime $p$ is good. For type $A_2$
this means that $p \ge 3$. We return to the case when $p=2$ in
Subsection \ref{subsec:A2-p2}.

\subsection{(Type $A_2$, $p \ge 3$)}
\label{subsec:A2}

Let $G = SL(3)$ with $p\geq 3$, and $\lambda = (\lambda_1, \lambda_2)$
be a dominant integral weight expressed in terms of the fundamental
weights.  Let us identify $\Delta = \{ \alpha, \beta \}$ and
$W = \{ e, s_\alpha, s_\beta, s_\alpha s_\beta, s_\beta s_\alpha,
s_\alpha s_\beta s_\alpha \}$. The cases where $\ell(w) \ne 2$,
i.e.,  $w \in \{ e, s_\alpha, s_\beta, s_\alpha s_\beta s_\alpha \}$,
are covered by Corollary \ref{cor:w-is-wI}. For such a $w$, set $\V =
\V_{B_1}( H^0(w, \lambda ) )$. We summarize in Table \ref{tab:A2}  below.

\begin{table}[ht]
	\centering
	\begin{tabular}{c|c|l} 
		$w$ &  $\V_{B_1}(H^0(w,\lambda )$ & $\lambda$ \\[2pt]
		\hline 
		$e$ & ${\mathfrak u}$ & \text{all} $\lambda$ \\[2pt] 
		$s_\alpha$ & ${\mathfrak u}_\alpha$ & 
		$p \mid \lambda_1 + 1$ \\[2pt]
		& ${\mathfrak u}$ & $p \nmid \lambda_1 + 1$ \\[2pt]
		$s_\beta$ & ${\mathfrak u}_\beta$ & $p \mid \lambda_2 + 1$ \\[2pt]
		& ${\mathfrak u}$ & $p \nmid \lambda_2 + 1$ \\[2pt]
		$s_\alpha s_\beta s_\alpha$ & $\V_{G_1}(H^0(\lambda)) \cap {\mathfrak u}$ & 
		all $\lambda$ \\[2pt]
	\end{tabular}
	\vspace{0.2cm}
	\caption{$B_1$-support varieties for $A_2$ when $\ell(w) \ne 2$, $p\geq 3$}
	\label{tab:A2}
\end{table} 
In the $w = s_\alpha s_\beta s_\alpha$ case, $J \subset \Delta$
depends on $\lambda$ and $p$ as in the discussion before Corollary
\ref{cor:w-is-wI}.

For the cases where $\ell(w) = 2$, we analyze the regularity of $\lambda$ with respect to the prime 
$p$ and $p$-divisibility of the dimension of $\hwl$. We treat the case $w = s_\alpha s_\beta$, the other case 
being symmetric upon switching $\alpha$, $\beta$ and $\lambda_1, \lambda_2$. For 
convenience, denote by $M(\lambda) = M(\lambda_1, \lambda_2)$ the $B$-module
$H^0(s_\alpha s_\beta, \lambda)$ which we also identify with 
$\ind_B^{P_\alpha} \ind_B^{P_\beta} \lambda$.

In the root system of type $A_2$, a weight $\lambda$ is $p$-regular if
and only if
\begin{equation*}
\label{eq:A2-regularity}
\tag{A}
\left\{ \begin{array}{l}
	p \nmid \lambda_1 + 1, \\
	p \nmid \lambda_2 + 1, \\
	p \nmid \lambda_1 + \lambda_2 + 2 .\end{array} \right. 
\end{equation*}

We may apply the Demazure character formula (\cite{A}) in this situation to conclude that $\dim M(\lambda) = \frac{(\lambda_2+1)(2 \lambda_1 + \lambda_2 + 2)}{2}$. Thus $p$ does not divide $\dim M(\lambda)$ if and only if
\begin{equation*}
\label{eq:A2-dimension}
\tag{B}
\left\{
\begin{array}{l}
p \nmid \lambda_2+1, \\
p \nmid 2 \lambda_1 + \lambda_2 + 2 .
\end{array} \right. 
\end{equation*}

\begin{thm}
\label{thm:A2}
Let $p\geq 3$. The $B_1$-support variety $\V = \V_{B_1}(M(\lambda))$ is 
${\mathfrak u}$ if either \eqref{eq:A2-regularity} or \eqref{eq:A2-dimension} hold. Otherwise $\V$ is a proper subvariety of ${\mathfrak u}$ given by the conditions below:
\begin{equation}
\label{eq:A2-length2-calculation}
\V_{B_1}(M(\lambda)) = \left\{
\begin{array}{ll}
{\mathfrak u}_\alpha, & \text{if}\,\, \lambda = (np - 1, 0) \quad (n \ge 1), \\[3pt]
{\mathfrak u}_\alpha \cup {\mathfrak u}_\beta, & \text{if}\,\, \lambda_2 \ne 0 \,\text{and neither 
\eqref{eq:A2-regularity} nor \eqref{eq:A2-dimension} hold.} \\[3pt]
\end{array} \right. 
\end{equation}
\end{thm}

The rest of the section is devoted to proving Theorem \ref{thm:A2}. First, if either 
\eqref{eq:A2-regularity} or \eqref{eq:A2-dimension} holds, 
we conclude that $\V = {\mathfrak u}$ by Proposition \ref{prop:regular-case} or the rank variety 
description of $\V$, respectively. For the rest of the section we assume that neither \eqref{eq:A2-regularity} 
nor \eqref{eq:A2-dimension} holds and calculate $\V$ which will turn out to be a proper 
subvariety of ${\mathfrak u}$.

Our analysis uses the $B$-stability of support varieties in a crucial way, in
particular the action of positive root subgroups and certain one-parameter
groups in the maximal torus. The nilradical $\mathfrak{u}$ is spanned by root
spaces $\mathfrak{u} = k X_\alpha \oplus k X_\beta \oplus k X_{\alpha+\beta}$.
There is a one-parameter subgroup $k^* \subset T \subset B$ such that \[
t.X_\gamma = t^{\operatorname{ht}(\gamma)} X_\gamma, \] for all $t \in k^*$
and $\gamma \in \Phi$. This group is generated by the element usually denoted
by $H_\rho$, where $\rho$ is the half sum of positive roots ($\rho = \alpha +
\beta$ in this case).

As a preliminary step we classify the $(B,k^*)$-stable subvarieties of
$\mathfrak{u}$. Let $v = a X_\alpha + b X_\beta + c X_{\alpha+\beta}$ be an
arbitrary point of $X \subset \mathfrak{u}$ an irreducible $B$-stable
subvariety of $\mathfrak{u}$. Here, $\rk v$ denotes the rank of a matrix
representative for $v$. The claim is that $X$ is equal to one of the following
($B$-stable) subspaces: $\mathfrak{u}$, $\overline{B \cdot X_\alpha} = k
X_\alpha \oplus k X_{\alpha+\beta}$, $\overline{B \cdot X_\beta} = k X_\beta
\oplus k X_{\alpha+\beta}$, and $\overline{B \cdot X_{\alpha+\beta}} = k
X_{\alpha+\beta}$, or $\{0\}$. There are five mutually exclusive cases:
\begin{enumerate} 
\item Suppose that $a,b \ne 0$. Then $\rk v = 2$ and the
$B$-orbit through $v$ is dense in $\mathfrak{u}$. Thus $X = \mathfrak{u}$.
\item Suppose that $a \ne 0, b = 0$. Using the action of $k^*$, we see that
the element $v' = a X_\alpha$ is in the closure of $B \cdot v$. Hence, 
$X_\alpha \in X$ and so $\overline{B \cdot X_\alpha} \subset X$. 
\item Suppose
that $a = 0, b \ne 0$. Then as in the previous case we conclude that
$\overline{B \cdot X_\beta} \subset X$. 
\item Suppose that $a, b = 0$ and $c
\ne 0$. In this case $\overline{B \cdot X_{\alpha + \beta}} \subset X$. 
\item Suppose that $a,b,c = 0$. Then, $v = 0$ and $\{0\} \subset X$.
\end{enumerate} 
Therefore, every $B$-stable, irreducible subvariety $X
\subset {\mathfrak u}$ is a union of the five subspaces above, thus it must
equal one of them.

Now we treat a number of cases depending on $\lambda$ and $p$ to determine
which root vectors are in the support variety. By the analysis of the previous
paragraph, this suffices to determine the variety as a union of $B$-stable
subvarieties.

First, suppose $\lambda_2 = 0$. In this case, $\ind_B^{P_\beta} \lambda \cong \lambda$ as a $B$-module and so $M(\lambda_1, 0) \cong \ind_B^{P_\alpha} (\lambda_1, 0)$. Thus by the $\ell(w) = 1$ calculation in Table \ref{tab:A2},
\[ \V_{B_1}(M(\lambda_1,0)) = \left\{ \begin{array}{ll} 
{\mathfrak u}_\alpha, & \text{if}\,\, p \mid \lambda_1 + 1 \\[3pt]
{\mathfrak u}, & \text{if}\,\, p \nmid \lambda_1+1 .\\[3pt]
\end{array} \right.
\]
This proves the first part of \eqref{eq:A2-length2-calculation}. 
Note that if $p$, $\lambda$ are such that $p \nmid \lambda_1+1$ and $\lambda_2 = 0$,
then they satisfy both \eqref{eq:A2-regularity} and \eqref{eq:A2-dimension}.

Now suppose $\lambda_2 \ne 0$. $M(\lambda)$ is induced from $H^0(s_\beta, \lambda)$ as a 
$P_\alpha$-module and as an
$L_\alpha$-module we have: \[ M(\lambda)\lvert_{L_\alpha} \cong \bigoplus_{i=1}^{\lambda_2+1}
\ind^{L_\alpha}_{B \cap L_\alpha} (\lambda_1+i) ,\] where the right hand side
is a direct sum of irreducible $L_\alpha$-modules indexed by the integers
$\lambda_1+1, \cdots, \lambda_1 + \lambda_2 + 2$. The assumption that $\lambda_2
\ne 0$ implies that $M(\lambda)$ has at least two $L_\alpha$ summands whose dimensions
differ by $1$ and so it cannot be projective over $\langle X_\alpha \rangle$. Therefore, 
$X_\alpha \in \V$ and by the analysis above of the $B$-stable, conical subvarieties of ${\mathfrak u}$ 
we get $\overline{B \cdot X_\alpha}
\subset \V$. Thus $X_{\alpha+\beta} \in \overline{B \cdot X_\alpha} \subset
\V$. Using the fact that $\V$ is $P_\alpha$-stable we can conclude that
$X_\beta \in \overline{P_\alpha \cdot X_{\alpha + \beta}} \subset \V$. Hence,
independent of $p$ we have: 
\begin{equation*} 
\overline{B \cdot X_\alpha} \cup \overline{B
\cdot X_\beta} = {\mathfrak u}_\alpha \cup {\mathfrak u}_\beta \subseteq \V.
\end{equation*} 

Suppose that $p \mid(\lambda_2+1)$. In this case, $\lambda_2$ is a Steinberg weight for $L_\beta$
and we have $\V_{B_1}(H^0(s_\beta, \lambda )) = {\mathfrak u}_\beta$. By Theorem \ref{thm:fp-bend-generalization}, 
\[ \V_{(P_\alpha)_1}(M(\lambda)) \subseteq P_\alpha \cdot \V_{B_1}(H^0(s_\beta,
\lambda )) = P_\alpha \cdot {\mathfrak u}_\beta .\] Now,
observe that the right hand side is contained in the subvariety \[ R_1 :=
\left\{ v \in \mathfrak{g} \mid \rk v \le 1 \right\} .\] Since $\rk (X_\alpha
+ X_\beta) = 2$ we have that $(X_\alpha + X_\beta) \notin \V$ so $\V$ is a
proper subvariety of $\mathfrak{u}$. We conclude in this case that \[ \V =
{\mathfrak u}_\alpha \cup {\mathfrak u}_\beta .\]

\subsection{Properness of Supports} We continue the proof of Theorem \ref{thm:A2} with all assumptions from Subsection \ref{subsec:A2}; in particular, $w=s_\alpha s_\beta$. 
Now assuming that $\lambda_2 \ne 0$ and $p \nmid \lambda_2+1$, we reduce to 
two families of modules which also satisfy neither \eqref{eq:A2-regularity} nor \eqref{eq:A2-dimension}.
\begin{lem}
\label{lem:type-i-ii}
If $\lambda = (\lambda_1, \lambda_2)$ satisfy $\lambda_2 \ne 0$, $p \nmid \lambda_2+1$, and neither \eqref{eq:A2-regularity} nor \eqref{eq:A2-dimension}, then either
\begin{equation*}
\label{eq:type1}
\tag{i}
\lambda_1 \equiv -1 \pmod{p}, \quad \lambda_2 \equiv 0 \pmod{p}
\end{equation*}
or
\begin{equation*}
\label{eq:type2}
\tag{ii}
\lambda_1 \equiv 0 \pmod{p}, \quad \lambda_2 \equiv -2 \pmod{p}
\end{equation*}
\end{lem}
\begin{proof}
First, if $\lambda$, $p$ violate condition \eqref{eq:A2-dimension} and $p \nmid 
\lambda_2+1$, then $p \mid 2\lambda_1+\lambda_2+2$. Since the pair also violates \eqref
{eq:A2-regularity}, there are two possibilities.
\begin{enumerate}
\item Suppose that $p \mid \lambda_1 + 1$. Then $p \mid (2\lambda_1+\lambda_2+2) = 2
(\lambda_1+1) + \lambda_2$ if and only if $p \mid \lambda_2$. This is case \eqref{eq:type1} 
above.
\item Suppose that $p \mid \lambda_1 + \lambda_2 + 2$. Then $p \mid (2\lambda_1+
\lambda_2+2) = (\lambda_1+\lambda_2+2) + \lambda_1$ if and only if $p \mid \lambda_1$. 
Hence, $p \mid \lambda_2+2$. These two conditions are equivalent to case \eqref{eq:type2} 
above.
\end{enumerate}
\end{proof}

To complete the proof, we make two reductions. First, we prove that if the support variety 
for all modules of type \eqref{eq:type1} in Lemma~\ref{lem:type-i-ii} are proper then the support varieties for all 
modules of type \eqref{eq:type2} are proper, and vice versa. Next, we show by induction 
that it suffices to prove the properness of the support variety for modules of the form 
$M(np, p-2)$ for $n \ge 0$. These are modules of type \eqref{eq:type2}. Finally, we analyze 
the support of $M(np, p-2)$ using filtrations on the tensor product $M(np, mp-2) \otimes L
(0,1)^{(1)}$ where $L(0,1)^{(1)}$ denotes the $G$-module $L(0,1)$ (with highest weight $\mu 
= (0,1)$) twisted once by the Frobenius morphism.
\begin{lem}
\label{lem:typei-typeii}
The support $\V_{B_1}(M(\lambda))$ is a proper subvariety of ${\mathfrak u}$ (and hence equal to ${\mathfrak u}_\alpha 
\cup {\mathfrak u}_\beta$) for all $\lambda$ of type (i) if and only if the same holds for all $\lambda$ of type (ii).
\end{lem}
\begin{proof}
Consider the $B$-module $M(np, mp-1)$ for some $n \ge 0$, $m > 0$. This module has proper support by the argument given for the case $p \mid \lambda_2+1$. Let $L(1,0)$ denote the irreducible $G$-module with highest weight $(1,0)$ and consider the tensor product $M(np, mp-1) \otimes L(1,0)$. The $G$-module structure on $L(1,0)$ allows to use the tensor identity (\cite[I.4.8]{Jan}) to identify
\begin{align*}
M(np, mp-1) \otimes L(1,0) & = \left[ \ind_B^{P_\alpha} \ind_B^{P_\beta} (np, mp-1) \right] \otimes L(1,0) \\
& \cong \ind_B^{P_\alpha} \ind_B^{P_\beta} \left[ 
(np, mp-1) \otimes L(1,0)_{\vphantom{B}}^{\vphantom{P_\alpha}} \right] .
\end{align*}
Now, $L(1,0)$ has a filtration as a $B$-module as follows
\[ 
\begin{array}{ccc}
L(1,0) = \left[ 
\vcenter{
\def\objectstyle{\scriptstyle}
\def\labelstyle{\scriptstyle}
\xymatrix{
(1,0) \ar@{-}[d] \\
(-1,1) \ar@{-}[d] \\
(0,-1)
}} \right. &
\text{which induces} &
(np,mp-1) \otimes L(1,0) = \left[
\vcenter{
\def\objectstyle{\scriptstyle}
\def\labelstyle{\scriptstyle}
\xymatrix{
(np+1,mp-1) \ar@{-}[d] \\
(np-1,mp) \ar@{-}[d] \\
(np,mp-2)
}} \right. .
\end{array}
\]
Let $F(\cdot)$ denote the functor $\ind_{B}^{P_\alpha} \ind_B^{P_\beta} (\cdot)$. Since 
the weights in the filtration for \mbox{$(np,mp-1) \otimes L(1,0)$} are all dominant, Kempf's 
vanishing theorem implies that $R^1F(\cdot)$ vanishes on each of the subquotients 
(\cite[I.4.4]{Jan}). Thus there is an induced filtration:
\[ M(np,mp-1) \otimes L(1,0) = \left[
\vcenter{
\def\objectstyle{\scriptstyle}
\def\labelstyle{\scriptstyle}
\xymatrix{
M(np+1,mp-1) \ar@{-}[d] \\
M(np-1,mp) \ar@{-}[d] \\
M(np,mp-2)
}} \right. .\]
Let $N$ denote the quotient $( M(np,mp-1) \otimes L(1,0) ) / M(np,mp-2)$. We have an 
exact sequence
\[ 0 \to M(np, mp-2) \to M(np,mp-1) \otimes L(1,0) \to N \to 0 .\]
The support variety of the middle term $M(np,mp-1) \otimes L(1,0)$ is proper. Now, $N$ 
sits in an exact sequence:
\[ 0 \to M(np-1, mp) \to N \to M(np+1,mp-1) \to 0 .\]
The support variety of the last term in this sequence, $M(np+1,mp-1)$, is proper by the $p \mid \lambda_2+1$ case.

Thus if the support of $M(np, mp-2)$ is proper, then the same hold for $N$ 
and hence by the second sequence, the same holds for $M(np-1, mp)$. On the other hand, 
if the support of $M(np-1, mp)$ is proper the second sequence 
implies that the same holds for $N$ and thus, by the first sequence, the same holds for $M(np, mp-2)$ 
(cf. \cite[(2.2.7)]{NPV} for properties of support varieties and exact sequences). 
\end{proof}

\begin{lem}
\label{lem:reduce-to-base}
The support variety $\V_{B_1}(M(np, mp-2))$ is proper for all $n \ge 0$, $m > 0$ if $
\V_{B_1}(M(np, p-2))$ is proper for all $n \ge 0$.
\end{lem}
\begin{proof}
The results follow by induction on $m$. Suppose that $\V_{B_1}(M(np, kp-2))$ is proper 
for all $n \ge 0$ and all $0 \le k \le m$. We prove that $\V_{B_1}(M(np,(m+1)p-2))$ is 
proper. Consider the tensor product
$M(np, mp-2) \otimes L(0,1)^{(1)}$. As in Lemma 
\ref{lem:typei-typeii}, we use the tensor identity and a filtration on $L(0,1)^{(1)}$. 
We have
\[ 
\begin{array}{ccc}
L(0,1)^{(1)} = \left[ 
\vcenter{
\def\objectstyle{\scriptstyle}
\def\labelstyle{\scriptstyle}
\xymatrix{
(0,p) \ar@{-}[d] \\
(p,-p) \ar@{-}[d] \\
(-p,0)
}} \right. &
\Longrightarrow &
M(np,mp-2) \otimes L(0,1)^{(1)} = \left[
\vcenter{
\def\objectstyle{\scriptstyle}
\def\labelstyle{\scriptstyle}
\xymatrix{
M(np,(m+1)p-2) \ar@{-}[d] \\
M((n+1)p-1,(m-1)p-2) \ar@{-}[d] \\
M((n-1)p,mp-2)
}} \right.
\end{array} .
\]
Let $N$ be the submodule such that $(M(np,mp-2) \otimes L(0,1)^{(1)}) / N 
\cong M(np,(m+1)p-2)$. The filtration on $N$ has subquotients whose 
supports are proper by the induction hypothesis, hence $N$ has proper 
support. It follows that $M(np,(m+1)p-2)$ has proper support.
\end{proof}
Finally, we prove that modules of the form $M(np, p-2)$ have proper 
support. This will finish off the calculation for $l(w)=2$ when $\Phi=A_{2}$ with $p\geq 3$. 

\begin{lem}
\label{lem:final-reduction}
The support variety $\V_{B_1}(M(np, p-2))$ is proper, hence
\[ \V_{B_1}(M(np, p-2)) = {\mathfrak u}_\alpha \cup {\mathfrak u}_\beta .\]
\end{lem}
\begin{proof}
We argue by induction on $n$. The base case is $M(0,p-2)$. This module has proper support by Corollary \ref{thm:w-is-w0J} since $\Phi_{\lambda, p} = 
\{ \alpha + \beta \}$. Assume that $M(kp,p-2)$ has proper support for 
all $0 \le k \le n$. We show that $M((n+1)p,p-2)$ has proper support.

As in the previous two lemmas, we consider a tensor product, in this case
$M(np,p-2) \otimes L(1,0)^{(1)}$. The filtration on $(np,p-2) \otimes L(1,0)^{(1)}$ now 
has socle the 1-dimensional $B$-module $(np,-2)$ which is not a dominant weight so we 
are forced to consider the higher derived functors $R^iF$, $i > 0$.

The $G$-module $L(1,0)^{(1)}$ has a $B$-filtration with sections of the form: 
\[ 
L(1,0)^{(1)} = \left[ 
\vcenter{
\def\objectstyle{\scriptstyle}
\def\labelstyle{\scriptstyle}
\xymatrix{
(p,0) \ar@{-}[d] \\
(-p,p) \ar@{-}[d] \\
(0,-p)
}} \right. .
\]
Tensoring with $(np, p-2)$ gives an exact sequence of $B$-modules:
\[ 0 \to (np, -2) \to L(1,0)^{(1)} \otimes (np, p-2) \to 
\left[ \vcenter{
\def\objectstyle{\scriptstyle}
\def\labelstyle{\scriptstyle}
\xymatrix{ (p,0) \ar@{-}[d] \\ (-p, p) }}\right. \otimes (np, p-2) \to 0 .\]
Applying the induction functor $F(\cdot)$ we have a long exact sequence in cohomology:
\begin{equation}
\label{eq:les1}
\xymatrix{ 
0 \ar[r] & F(np,-2) \ar[r] & L(1,0)^{(1)} \otimes F(np, p-2) \ar[r] & 
F\left(
\left[
\def\objectstyle{\scriptstyle}
\def\labelstyle{\scriptstyle}
\xy
{\ar@{-} (0,5)*{(p,0)}; (0,-5)*{(-p,p)}}; 
\endxy
\right. \otimes (np, p-2) 
\right) \\
\ar[r] & \ar[r] R^1F(np, -2) \ar[r] & 0 .& }
\end{equation}
The first term, $F((np,-2))$, vanishes since $(np,-2)$ is not $\beta$-dominant, so \eqref{eq:les1} is a 
short exact sequence. Also note that 
the second term has proper support by the induction hypothesis. Now, we claim that 
the module $R^1F((np,-2))$ has proper support. 

Recall that $F(\cdot) = \ind_B^{P_\alpha} \circ \ind_B^{P_\beta}( \cdot )$. Consider the 
spectral sequence: 
$$E_{2}^{i,j}=R^{i}\text{ind}_{B}^{P_{\alpha}}R^{j}\text{ind}_{B}^{P_{\beta}}(np,-2)\Rightarrow 
R^{i+j}F(np,-2).$$ 
Set $E_1 = R^1F((np,-2))$. The spectral sequence yields a five term exact sequence of the form:
\[
\xymatrix{
0 \ar[r] & R^1\ind_B^{P_\alpha} \left( \ind_B^{P_\beta}( (np, -2) ) \right) \ar[r] &
E_1  & \\
\ar[r] & \ind_B^{P_\alpha}\left( R^1\ind_B^{P_\beta}( (np, -2) ) \right) \ar[r] &
R^2\ind_B^{P_\alpha}\left( \ind_B^{P_\beta}( (np,-2) ) \right) \ar[r] & \cdots .
}
\]
Since $\ind_B^{P_\beta}( (np,-2) ) = 0$, the first and last term vanish so we have
$$E_{1} \cong \ind_B^{P_\alpha}\left( R^1\ind_B^{P_\beta}( (np, -2) ) \right).$$ 
By Serre Duality (\cite[Prop. 5.2(c)]{Jan}), $R^1\ind_B^{P_\beta}( (np, -2) ) \cong (np-1,0)$. Consequently, 
from the $\ell(w) = 1$ case we can conclude that $E_1 \cong \ind_B^{P_\alpha}( (np-1,0) )$ has proper support.

Now \eqref{eq:les1} implies that the module
\[
F\left(
\left[
\def\objectstyle{\scriptstyle}
\def\labelstyle{\scriptstyle}
\xy
{\ar@{-} (0,5)*{(p,0)}; (0,-5)*{(-p,p)}}; 
\endxy
\right. \otimes (np, p-2) 
\right)
\]
has proper support. We have an exact sequence:
\[
0 \to F((n-1)p,2p-2) \to F\left( \left[
\def\objectstyle{\scriptstyle}
\def\labelstyle{\scriptstyle}
\xy
{\ar@{-} (0,5)*{(p,0)}; (0,-5)*{(-p,p)}}; 
\endxy
\right. \otimes (np, p-2) \right)
\to F((n+1)p, p-2) \to 0 .\]
The middle term has proper support and we to show that the last term has proper support. Thus it suffices to show that $F( ((n-1)p,2p-2) )$ has proper support. This is the other base case in our double induction.

We argue as in Lemma \ref{lem:reduce-to-base} with $n$ replaced by $n-1$ and $m=1$. Now, the $B$-filtration on $((n-1)p,p-2) \otimes L(0,1)^{(1)}$
has a non-dominant weight in the middle layer $((n+1)p,-2)$. Let $N$ denote the quotient $N := (((n-1)p,p-2) \otimes L(0,1)^{(1)})/((n-2)p,p-2)$ so that $N$ has socle consisting of $((n+1)p, -2)$. We have an exact sequence
\begin{equation}
\label{eq:FN}
0 \to F( (n-2)p, p-2 ) \to F\left[ ((n-1)p,p-2) \otimes L(0,1)^{(1)} 
\right] \to F(N) \to 0
\end{equation}
with first and middle terms having proper support, thus $F(N)$ has proper support. Furthermore, $F(N)$ sits in a sequence
\[ 0 \to ((n+1)p, -2) \to N \to ((n-1)p, 2p-2) \to 0 \]
and applying $F(\cdot)$ we have
\[ 
\xymatrix{
0 \ar[r] & F((n+1)p, -2) \ar[r] & F(N) \ar[r] & F((n-1)p, 2p-2) \\
      \ar[r]   & R^1F((n+1)p, -2) \ar[r] & R^1F(N) \ar[r] & 0 .}
\]
The term $F((n+1)p, -2)$ vanishes and so does $R^1F(N)$ by extending 
the sequence \eqref{eq:FN}. As before, we identify $R^1F((n+1)p,-2) \cong 
\ind_B^{P_\alpha}((n+1)p-1,0)$ which has proper support. Thus the term 
$F((n-1)p, 2p-2)$ has proper support and the proof is concluded.
\end{proof}
\vskip 1cm 

\subsection{(Type $A_{2}$, $p=2$)} 
\label{subsec:A2-p2}

When $p=2$, ${\mathcal N}_{1}({\mathfrak u})={\mathfrak u}_{\alpha}\cup {\mathfrak u}_{\beta}$. One can 
apply results from Sections 4 and 5 to give explicit descriptions of the $B_{1}$-supports of 
$H^{0}(w,\lambda)$ when $l(w)\neq 2$. We summarize the results in Table \ref{tab:A2p=2}.

\begin{table}[!h]
	\centering
	\begin{tabular}{c|c|l} 
		$w$ &  $\V_{B_1}(H^0(w,\lambda ))$ & $\lambda$ \\[2pt]
		\hline 
		$e$ & ${\mathfrak u}_{\alpha}\cup {\mathfrak u}_{\beta}$ & \text{all} $\lambda$ \\[2pt] 
		$s_\alpha$ & ${\mathfrak u}_\alpha$ & 
		$p \mid \lambda_1 + 1$ \\[2pt]
		& ${\mathfrak u}_{\alpha}\cup {\mathfrak u}_{\beta}$ & $p \nmid \lambda_1 + 1$ \\[2pt]
		$s_\beta$ & ${\mathfrak u}_\beta$ & $p \mid \lambda_2 + 1$ \\[2pt]
		& ${\mathfrak u}_{\alpha}\cup {\mathfrak u}_{\beta}$ & $p \nmid \lambda_2 + 1$ \\[2pt]
		$s_\alpha s_\beta s_\alpha$ & $\V_{G_1}(H^0(\lambda)) \cap ({\mathfrak u}_{\alpha}\cup {\mathfrak u}_{\beta})$ & 
		all $\lambda$ \\[2pt]
	\end{tabular}
	\vspace{0.2cm}
	\caption{$B_1$-support varieties for $A_2$ when $\ell(w) \ne 2$, $p=2$.}
	\label{tab:A2p=2}
\end{table} 

In the $l(w)=2$ case it suffices to consider (by symmetry) $w=s_{\alpha}s_{\beta}$. We note that 
there are no $p$-regular weights and $\dim M(\lambda)$ is always divisible by 2. Moreover, 
we don't need to show properness because all $B_{1}$-support varieties are already contained in 
${\mathfrak u}_{\alpha}\cup {\mathfrak u}_{\beta}$. The following result summarizes the $l(w)=2$ case. 

\begin{thm} Let $w=s_{\alpha}s_{\beta}$. The $B_1$-support variety $\V = \V_{B_1}(M(\lambda))$ is given by: 
\begin{equation}
\label{eq:A2-length2-calculationp=3}
\V_{B_1}(M(\lambda)) = \left\{
\begin{array}{ll}
{\mathfrak u}_\alpha, & \text{if}\,\, \lambda = (2n- 1, 0) \quad (n \ge 1), \\[3pt]
{\mathfrak u}_\alpha \cup {\mathfrak u}_\beta, & \text{if}\, \, \lambda=(2n,0) \quad (n \geq 0),  \\[3pt]
{\mathfrak u}_\alpha \cup {\mathfrak u}_\beta, & \text{if}\,\, \lambda_2 \ne 0  .\\[3pt]
\end{array} \right. 
\end{equation}
\end{thm}

\end{document}